\documentclass[12pt,twoside]{amsart}

\usepackage{amssymb}
\usepackage{graphicx}
\usepackage[pdfauthor={David Wayne Cook II},
            pdftitle={Cohen-Macaulay graphs and face vectors of flag complexes},
            pdfproducer={LaTeX with hyperref},
            pdfcreator={latex->dvips->ps2pdf},
            pdfborder={0 0 0},
            colorlinks=false]{hyperref}  
\usepackage[margin=1in]{geometry}
\newcommand{\al}{\alpha}
\newcommand{\be}{\beta}

\newcommand{\Ga}{\Gamma}
\newcommand{\De}{\Delta}
\newcommand{\si}{\sigma}
\newcommand{\Si}{\Sigma}
\newcommand{\ta}{\tau}

\newcommand{\st}{\; \mid \;}

\DeclareMathOperator{\Ind}{Ind}
\DeclareMathOperator{\link}{link}
\DeclareMathOperator{\fdel}{del}

\numberwithin{figure}{section}
\numberwithin{equation}{section}

\newtheorem{theorem}{Theorem}[section]
\newtheorem{lemma}[theorem]{Lemma}
\newtheorem{proposition}[theorem]{Proposition}
\newtheorem{corollary}[theorem]{Corollary}
\newtheorem{conjecture}[theorem]{Conjecture}

\theoremstyle{definition}

\newtheorem{example}[theorem]{Example}

\newtheorem*{acknowledgement}{Acknowledgement}

\begin{document}

\title{Cohen-Macaulay graphs and face vectors of flag complexes}
\author[D.\ Cook II, U.\ Nagel]{David Cook II, Uwe Nagel}
\address{Department of Mathematics, University of Kentucky, 715 Patterson Office Tower, Lexington, KY 40506-0027, USA}
\email{dcook@ms.uky.edu, uwe.nagel@uky.edu}
\thanks{Part of the work for this paper was done while the second author was partially supported by the National Security Agency under Grant Number H98230-09-1-0032.}

\begin{abstract}
    We introduce a construction on a flag complex that, by means of modifying the associated graph, generates a new flag complex whose
    $h$-factor is the face vector of the original complex.  This construction yields a vertex-decomposable, hence Cohen-Macaulay, 
    complex.  From this we get a (non-numerical) characterisation of the face vectors of flag complexes and deduce also that the
    face vector of a flag complex is the $h$-vector of some vertex-decomposable flag complex.  We conjecture that the converse of
    the latter is true and prove this, by means of an explicit construction, for $h$-vectors of Cohen-Macaulay flag 
    complexes arising from bipartite graphs.  We also give several new characterisations of bipartite graphs with Cohen-Macaulay
    or Buchsbaum independence complexes.
\end{abstract}

\maketitle

\section{Introduction} \label{sec:introduction}

Simplicial complexes are combinatorial objects at the intersection of many fields of mathematics including algebra and topology.
Passing from a simplicial complex to its barycentric subdivision yields a flag complex---a simplicial complex whose minimal non-faces
are edges---while preserving topological properties. This, among other reasons, lead Stanley to state~\cite[p.\ 100]{St} that ``Flag
complexes are a fascinating class of simplicial complexes which deserve further study.''  One simple enumeration of a simplicial
complex is the face vector. Face vectors are conveniently expressed as $h$-vectors which admit a more algebraic interpretation (via 
Hilbert series).  An important task is to establish restrictions on the face or, equivalently, $h$-vectors of simplicial flag complexes.
In this note we contribute to this problem by exploring the interplay of face vectors and $h$-vectors of flag complexes.

According to~\cite[p.\ 100]{St}, Kalai conjectured that the face vector of a flag complex $\De$ is also the face vector of a balanced
complex $\Ga$.  Moreover, if $\De$ is Cohen-Macaulay, then $\Ga$ can be chosen to be Cohen-Macaulay.  If this conjecture is true (in its
entirety), then it would follow that the $h$-vector of a Cohen-Macaulay flag complex is the face vector of a balanced
complex. The first part of the conjecture, which does not assume Cohen-Macaulayness, was also conjectured by Eckhoff \cite{E} and has
recently been proven by Frohmader in~\cite[Theorem~1.1]{Fr}.  However, the second part of the conjecture remains open. Similarly,
\cite{BFS} discusses the relation between $h$-vectors of Cohen-Macaulay complexes and the face vectors of multi-complexes. In
this note, we begin studying the question of which Cohen-Macaulay flag complexes have $h$-vectors that are also the face vectors 
of flag complexes.

In Section~\ref{sec:preliminaries}, we recall some basic concepts used throughout the paper.  We introduce clique-whiskering, a
generalisation of whiskering graphs (see~\cite{DE}, \cite{FH}, \cite{LM}, \cite{SVV}, \cite{Vi}, and~\cite{Wo}), in
Section~\ref{sec:clique-whiskered}. An advantage of clique-whiskering is that it produces flag complexes of smaller
dimension than whiskering.  We show that the independence complex of a clique-whiskered graph is vertex-decomposable
(Theorem~\ref{thm:cw-vd}) and hence squarefree glicci (so, in particular, in the {\bf G}orenstein {\bf li}aison {\bf c}lass of a
{\bf c}omplete {\bf i}ntersection, see~\cite{NR}). This generalises results by Villarreal \cite{Vi} and Dochtermann and Engstr\"om
\cite{DE}. Moreover, we prove that the face vector of the independence complex of the base graph is the $h$-vector of the
independence complex of the clique-whiskered graph (Proposition~\ref{pro:cw-h-is-f}). This provides a characterisation
of the face vectors of flag complexes as the $h$-vectors of the independence complexes of clique-whiskered graphs
(Theorem~\ref{thm:cw-h-is-flag-f}) and also shows that the face vector of every flag complex is the $h$-vector of some
vertex-decomposable (hence Cohen-Macaulay) flag complex (Corollary~\ref{cor:cw-flag-fa}; compare also the independently-found result
in~\cite[Proposition~4.1]{CV}).  We conjecture that the converse of the latter is also true (Conjecture~\ref{con:h-is-f}).  As evidence,
we establish the conjecture in the case of independence complexes of bipartite graphs by means of another explicit construction
(Proposition~\ref{pro:bi-h-is-f}).

In Section~\ref{sec:bipartite}, we restrict ourselves to bipartite graphs.  We find another classification of bipartite graphs with
Cohen-Macaulay independence complexes.  From this, we establish that bipartite graphs with Buchsbaum independence complexes are exactly
those that are complete or have Cohen-Macaulay independence complexes (Theorem~\ref{thm:bi-Buchsbaum}); this result was found independently
in \cite{HYZ}.  Moreover, we define the compression of a bipartite graph and show that the $h$-vector of a Cohen-Macaulay flag complex
arising from a bipartite graph is the face vector of the independence complex of the associated compression (Proposition~\ref{pro:bi-h-is-f}).

After this note had been written, the paper~\cite{CV} of Constantinescu and Varbaro appeared which treats topics similar to those 
presented here, though in a greatly different manor.  We make more specific references in the main body of the text.

\section{Preliminaries} \label{sec:preliminaries}

A {\em simplicial complex} $\De$, on a finite set $V$, is a set of subsets of $V$ closed under inclusion; elements of $\De$ are called
{\em faces}.  The {\em dimension} of a face $\si$ is $\#\si - 1$ and of a complex $\De$ is the maximum of the dimensions of its faces.  A
complex whose maximal faces, called {\em facets}, are equi-dimensional is called {\em pure} and a complex with a unique facet is
called a {\em simplex}.

Let $\De$ be a $(d-1)$-dimensional simplicial complex.  There are two vectors and two sub-complexes of interest.  The {\em face vector} (or
$f$-vector) of $\De$ is the $(d+1)$-tuple $(f_{-1},\ldots, f_{d-1})$, where $f_i$ is the number of $i$-dimensional faces of $\De$.  The
{\em $h$-vector} of $\De$ is the $(d+1)$-tuple $(h_0, \ldots, h_d)$ given by $h_j = \sum_{i=0}^j(-1)^{j-i}\binom{d-i}{j-i}f_{i-1}$.
Notice that given the $h$-vector of a simplicial complex we can recover the face vector; indeed, for $0 \leq j \leq d$, one has
\begin{equation} \label{equ:h-to-f}
    f_{j-1} = \sum_{i=0}^j\binom{d-i}{j-i}h_{i}.
\end{equation}

Let $\si$ be a face of $\De$, then the {\em link} and {\em deletion} of $\si$ from $\De$ are given by
\[
    \link_\De{\si} := \{ \tau \in \De \st \tau \cap \si = \emptyset, \tau \cup \si \in \De \}
    \quad \text{and} \quad
    \fdel_\De{\si} := \{ \tau \in \De \st \si \nsubseteq \tau \}.
\]
Moreover, following~\cite[Definition~2.1]{PB}, a pure complex $\De$ is said to be to be {\em vertex-decomposable} if either $\De$ is a
simplex or there exists a vertex $v \in \De$, called a {\em shedding vertex}, such that both $\link_{\De}{v}$ and $\fdel_{\De}{v}$ are
vertex-decomposable.  Checking if a particular simplicial complex is vertex-decomposable can be achieved using a computer program such
as Macaulay2~\cite{M2}; the package described in~\cite{Co} provides the appropriate methods.

A {\em graph} $G = (V,E)$ is a pair consisting of a finite {\em vertex} set $V$ and an {\em edge} set $E$ of two-element subsets
of $V$.  Two vertices $u$ and $v$ are {\em adjacent} in $G$ if $uv \in E$, the {\em neighborhood} of a vertex $v$ is the set $N_G(v)$ of
vertices adjacent to $v$ in $G$, and the {\em degree} of a vertex is the cardinality of its neighborhood. A graph $G$ is {\em complete} if
every vertex is adjacent to every other vertex and a graph is {\em connected} if there is a path in $G$ between every pair of vertices of $G$.

Let $G = (V,E)$ be a graph.  A subset of vertices $U$ is an {\em independent set} if no two elements of $U$ are adjacent. The {\em
independence complex} of $G$ is the simplicial complex $\Ind{G}$ with faces generated by the independent sets of $G$.  Hence graphs
can be studied by looking at simplicial complexes.

The following result was first observed in~\cite{En} and follows directly from the definitions.
\begin{lemma} \label{lem:link-fdel}
    Let $v$ be a vertex of $G$, then
    \[
        \fdel_{\Ind{G}}{v} = \Ind(G \setminus v)
        \quad \text{and} \quad
        \link_{\Ind{G}}{v} = \Ind(G \setminus (v \cup N_G(v))).
    \]
\end{lemma}

These combinatorial objects are related to squarefree monomial ideals.  Let $\De$ be a simplicial complex on the set $V$.  The
{\em Stanley-Reisner ideal} of $\De$ is the ideal $I(\De)$ generated by the minimal non-faces of $\De$ and the {\em Stanley-Reisner ring}
of $\De$ is $K[\De] = K[V]/I(\De)$, for a field $K$.  Thus, the Stanley-Reisner ideals of simplicial complexes on some vertex set $V$
are exactly the squarefree monomial ideals in $K[V]$.

The Stanley-Reisner ideals of independence complexes of graphs are exactly the ideals in $K[V]$ generated by quadratic squarefree
monomials.  Moreover, the generating monomials correspond to the edges of the graphs. When a simplicial complex has a quadratic
Stanley-Reisner ideal, that is, it is the independence complex of some graph, then it is called a {\em flag} complex.

A simplicial complex is pure if and only if its Stanley-Reisner ring is unmixed, that is, has associated prime ideals that are
equi-dimensional.  If $(h_0, \ldots, h_d)$ is the $h$-vector of some $(d-1)$-dimensional complex $\De$, then the Hilbert series of
$K[\De]$ is given by
\[
    \frac{h_0 + h_1t + h_2t^2 + \cdots + h_dt^d}{(1-t)^d}.
\]
Moreover, if $\De$ and $\Si$ are complexes on disjoint vertex sets, then the Hilbert series of $\De \cup \Si$ is the product of the
Hilbert series of $\De$ and $\Si$.  Finally, a simplicial complex is {\em Cohen-Macaulay} (resp. {\em Buchsbaum}) if and only if the
associated Stanley-Reisner ring is Cohen-Macaulay (resp. Buchsbaum).

The following example demonstrates that not every face vector of a simplicial complex is the $h$-vector of a flag complex.

\begin{example} \label{exa:f-not-h}
    Consider the simplicial complex with facets $\{uv, uw, vw\}$; this complex has face vector $(1,3,3)$.  Suppose $G$ is a graph whose
    independence complex has an $h$-vector whose non-zero part is $(1,3,3)$ and is of dimension $d-1$, for some $d\geq2$.  Then the Hilbert 
    series of $K[\Ind{G}]$ is $\frac{1+3t+3t^2}{(1-t)^d}$.

    Assuming $G$ is on $n$ vertices, then $d = n - 3$ and so $n \geq 5$.  Further still, $G$ must have $3$ edges.  However, $1+3t+3t^2$ is
    irreducible over $\mathbb{Z}$ so $G$ must be connected; this is impossible when $n \geq 5$.  Thus no such $G$ can exist.
\end{example}

Notice however, in Theorem~\ref{thm:cw-h-is-flag-f}, we prove that the face vector of every flag complex is indeed the $h$-vector of another
flag complex.

\section{Clique-whiskered graphs} \label{sec:clique-whiskered}

Let $G = (V,E)$ be a (non-empty) graph with $V = \{v_1, \ldots, v_n\}$.

Adding a {\em whisker} to $G$ at $v_i$ means adding a new vertex $w$ and edge $v_iw$ to $G$.  It was shown in~\cite[Proposition
2.2]{Vi} that if a whisker is added to every vertex of $G$, then the resulting graph has a Cohen-Macaulay independence complex. 
Furthermore, in~\cite[Theorem 4.4]{DE} it was shown that the independence complex of a fully-whiskered graph is also pure and 
vertex-decomposable.  We give a generalisation of whiskering in this section and explore its properties.

A subset $C$ of the vertices is a {\em clique} if it induces a complete subgraph of $G$.  A {\em clique vertex-partition} of $G$ is a set
$\pi = \{W_1, \ldots, W_t\}$ of disjoint (possibly empty) cliques of $G$ such that their disjoint union forms $V$.  Notice that $G$ may
permit many different clique vertex-partitions, and every graph has at least one clique vertex-partition, in particular, the {\em
trivial partition}, $\ta = \{\{v_1\}, \ldots, \{v_n\}\}$.

{\em Clique-whiskering}, which \cite[Proposition~22]{Wo} called clique-starring, a clique $W$ of $G$ is done by adding a new vertex
$w$ and connecting $w$ to every vertex in $W$, resulting in the graph $G^W$.  We further define {\em fully clique-whiskering} $G$ by
a clique vertex-partition $\pi = \{W_1, \ldots, W_t\}$ to be $G$ clique-whiskered at every clique of $\pi$; it produces the graph
\[
    G^\pi := (V \cup \{w_1, \ldots, w_t\}, E \cup \{vw_i \st v \in W_i\}).
\]
Notice that $G^\ta$ is the fully-whiskered graph when $\ta$ is the trivial partition and empty cliques produce isolated vertices.

\begin{example} \label{exa:clique-whiskering}
    Let $G$ be the three-cycle on vertices $\{u,v,w\}$.  There are three distinct clique vertex-partitions of $G$ (without
    empty cliques): the trivial partition $\ta = \{\{u\},\{v\},\{w\}\}$, $\pi = \{\{u,v\},\{w\}\}$, and $\rho = \{\{u,v,w\}\}$.
    These are shown in Figure~\ref{fig:clique-whiskering}.
    \begin{figure}[!ht]
        \includegraphics{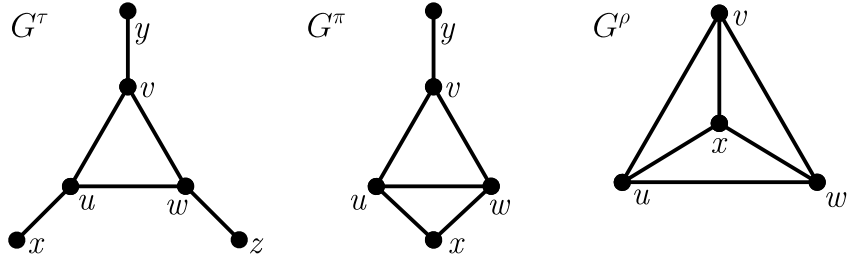}
        \caption{Clique-whiskerings of the three-cycle}
        \label{fig:clique-whiskering}
    \end{figure}
\end{example}

Following directly from the definition, we get purity of the independence complex.
\begin{lemma} \label{lem:cw-dim-purity}
    Let $\pi = \{W_1, \ldots, W_t\}$ be a clique vertex-partition of $G$.  Then $\Ind{G^\pi}$ is pure and
    $(t-1)$-dimensional.
\end{lemma}
\begin{proof}
    Let $w_1, \ldots, w_t$ be the new vertices associated to $W_1, \ldots, W_t$, respectively.

    Each clique-whisker $B_i = W_i \cup \{w_i\}$ is a clique of $G^\pi$, hence any independent set of $G^\pi$ has
    at most one vertex from each $B_i$, that is, $\dim{\Ind{G^\pi}} < t$.  Moreover, $\{w_1, \ldots, w_t\}$ is
    an independent set in $G^\pi$, as $N(w_i) = W_i$ are pairwise disjoint.  Thus $\dim{\Ind{G^\pi}} = t-1$.

    Let $I = \{u_1, \ldots, u_k\}$ be an independent set of $G^\pi$, and suppose, without loss of generality,
    that $u_j \in B_j$.  Then $I \cup \{w_{k+1}, \ldots, w_{t}\}$ is a maximal independent set in $G^\pi$ of size
    $t$.  Thus $\Ind{G^\pi}$ is pure.
\end{proof}

The following result generalises~\cite[Theorem 4.4]{DE}.  Further, it was shown in~\cite[Proposition~4.1]{CV} that fully-whiskered
graphs have balanced independence complexes; the proof easily extends to independence complexes of clique-whiskered graphs.
\begin{theorem} \label{thm:cw-vd}
    Let $\pi = \{W_1, \ldots, W_t\}$ be a clique vertex-partition of $G$.  Then $\Ind{G^\pi}$ is vertex-decomposable.
\end{theorem}
\begin{proof}
    Suppose $\pi = \{W_1, \ldots, W_t\}$ is a clique vertex-partition of $G$.  If, without loss of generality, $W_t$ is an empty
    clique, then $\rho = \{W_1, \ldots, W_{t-1}\}$ is a clique vertex-partition of $G$ and $\Ind{G^\rho}$ is a cone over
    $\Ind{G^\pi}$.  Recall that coning over a simplicial complex does not affect vertex-decomposability.

    Now, we proceed by induction on $n$, the number of vertices of $G$.  If $n=1$, then for any $\pi$ of $G$, $\Ind{G^\pi}$ is
    a pair of disjoint vertices (or a cone thereof), hence is vertex-decomposable.

    Let $n \geq 2$ and, without loss of generality, assume $v \in W_t$.  Define $\rho = \{W_1, \ldots, W_{t-1}, W_t \setminus {v}\}$,
    then $\rho$ is a clique vertex-partition of $G \setminus v$, a graph on $n-1$ vertices, and
    \[
        \fdel_{\Ind{G^\pi}}{v} = \Ind(G^\pi \setminus v) = \Ind((G\setminus v)^\rho)
    \]
    is vertex-decomposable by induction.  Define $\si = \{ W_1 \setminus N_G(v), \ldots, W_{t-1} \setminus N_G(v)\}$, then
    $\si$ is a clique vertex-partition of $G \setminus ({v} \cup N_G(v))$, a graph on $n - 1 - \# N_G(v)$ vertices, and
    \[
        \link_{\Ind{G^\pi}}{v} = \Ind(G^\pi \setminus ({v} \cup N_{G^\pi}(v)))
                               = \Ind((G \setminus ({v} \cup N_G(v)))^\si)
    \]
    is vertex-decomposable by induction.

    Thus $v$ is a shedding vertex of $\Ind{G^\pi}$ and $\Ind{G^\pi}$ is vertex-decomposable.
\end{proof}

Notice the proof of Theorem~\ref{thm:cw-vd} shows every vertex of $G$ is a shedding vertex of $\Ind{G^\pi}$.

By~\cite[Corollary~2.3]{SVV}, a fully-whiskered star graph is licci (in the {\bf li}aison {\bf c}lass of a {\bf c}omplete {\bf i}ntersection).
However, not every fully-whiskered graph has this property.

\begin{example} \label{exa:not-licci}
    Let $G$ be the full-whiskering of the three-cycle (see Figure~\ref{exa:clique-whiskering}, $G^\ta$) and $R = K[u,v,w,x,y,z]$.
    Then the Stanley-Reisner ring $K[\Ind{G}]$ has a free resolution of the form
    \[
        0 \longrightarrow R^3(-4) \longrightarrow R^8(-3) \longrightarrow R^6(-2) \longrightarrow R \longrightarrow K[\Ind{G}] \longrightarrow 0
    \]
    and hence, by~\cite[Corollary~5.13]{HU}, is not licci.
\end{example}

Being glicci is a weaker condition than being licci, though it is still true that every glicci ideal is Cohen-Macaulay.
It is one of the main open questions in liaison theory if every Cohen-Macaulay ideal is glicci. In this regard, we obtain:
\begin{corollary} \label{cor:cw-cm}
    Let $\pi$ be a clique vertex-partition of $G$.  Then $\Ind{G^\pi}$ is squarefree glicci and Cohen-Macaulay.
\end{corollary}
\begin{proof}
    Pure vertex-decomposable simplicial complexes are squarefree glicci~\cite[Theorem~3.3]{NR} and pure
    shellable~\cite[Theorem~2.8]{PB}, hence Cohen-Macaulay.
\end{proof}

The graphs that are full clique-whiskerings of graphs can be described in terms of clique vertex-partitions.  We note
that~\cite[Lemma~6.1]{CV} gives an equivalent statement regarding the flag complexes whose associated graphs are full
clique-whiskerings.
\begin{proposition} \label{pro:cw-which}
    A graph $G$ is a full clique-whiskering if and only if there exists a clique vertex-partition of $G$ such that
    every clique in the partition contains a vertex whose neighborhood in $G$ is contained in the clique.
\end{proposition}
\begin{proof}
    This follows directly from the definition of full clique-whiskering.
\end{proof}

Unfortunately, not every pure vertex-decomposable flag complex comes from the independence complex of a fully clique-whiskered
graph.  For example, the independence complex of the five-cycle is a pure vertex-decomposable flag complex, but the five-cycle is 
not the full clique-whiskering of any graph.

A pure simplicial complex $\De$ is called {\em partitionable} if it can be written as a disjoint union of intervals
$[G_1, F_1] \dot\cup \cdots \dot\cup [G_s, F_s]$, called a {\em partitioning} of $\De$, where $F_1,\ldots,F_s$ are the facets of 
$\De$ and $[G, F] := \{ H \st G \subseteq H \subseteq F\}$.  As pure vertex-decomposable simplicial complexes are shellable
(\cite[Theorem~2.8]{PB}), they are also partitionable (\cite[Statement before III.2.3]{St}).  Further, the $h$-vector of a 
partitionable simplicial complex can be written in terms of a given partitioning.
\begin{proposition}{\cite[Proposition~III.2.3]{St}} \label{pro:st-partition}
    Let $\De$ be a pure partitionable simplicial complex and let $[G_1, F_1] \dot\cup \cdots \dot\cup [G_s, F_s]$ be a partitioning 
    of $\De$.  Then the $h$-vector of $\De$ is given by $ h_i = \# \{j \st \# G_j = i\}$.
\end{proposition}

An interesting feature of clique-whiskered graphs is that the $h$-vectors of their independence complexes are the
same as the face vectors of the independence complexes of their base graphs.  This is can be seen as a generalisation
of~\cite[Theorem~2.1]{LM}, which provides the relation for the face vectors with respect to whiskering.
\begin{proposition} \label{pro:cw-h-is-f}
    Let $\pi$ be a clique vertex-partition of $G$.  Then the $h$-vector of $\Ind{G^\pi}$ is the face vector of $\Ind{G}$.
\end{proposition}
\begin{proof}
    The result can be obtained from the methods of the proof of~\cite[Theorem~6.3]{BN}; however, we prefer to give
    a more direct argument.

    Let $d = \#\pi = \dim{\Ind{G^\pi}} + 1$. There are $f_{i-1}(\Ind{G})$ independent sets of size $i$ in $G$; let $I$ be
    one of these independent sets.  Then there are exactly $d-i$ vertices added during clique-whiskering independent from
    $I$ in $G^\pi$.  Hence $I$ can be expanded to the independent set $\hat{I}$ of size $d$ in $G^\pi$ and so $\hat{I}$ is
    (with abuse of notation) a facet of $\Ind{G^\pi}$.  Moreover, we have a partitioning of $\Ind{G^\pi}$ given by the set
    of intervals $[I, \hat{I}]$, where $I$ runs through the independent sets of $G$. Hence, by Proposition~\ref{pro:st-partition},
    the $h$-vector of $\Ind{G^\pi}$ is the face vector of $\Ind{G}$.
\end{proof}

A consequence of the previous proposition, along with Proposition~\ref{pro:cw-which}, is a (non-numerical) characterisation of the
face vectors of flag complexes.
\begin{theorem} \label{thm:cw-h-is-flag-f}
    Let $f$ be a finite sequence of positive integers.  Then the following conditions are equivalent:
    \begin{enumerate}
        \item $f$ is the face vector of a flag complex,
        \item $f$ is the $h$-vector of the independence complex of a clique-whiskered graph, and
        \item $f$ is the $h$-vector of the independence complex of a fully-whiskered graph.
    \end{enumerate}
\end{theorem}

Another consequence is a (perhaps more useful) condition on the face vectors of flag complexes, which was established
independently with the added condition that the vertex-decomposable complex is also balanced (as noted above) 
in~\cite[Proposition~4.1]{CV}.
\begin{corollary} \label{cor:cw-flag-fa}
    The face vector of every flag complex is the $h$-vector of some vertex-decompos\-able flag complex.
\end{corollary}

We believe that the converse of Corollary~\ref{cor:cw-flag-fa} is true as well.  If so, this would provide a more complete
(non-numerical) characterisation of the face vectors of flag complexes.
\begin{conjecture} \label{con:h-is-f}
    The $h$-vector of every vertex-decomposable flag complex is the face vector of some flag complex.
\end{conjecture}
For evidence of this conjecture, see Proposition~\ref{pro:bi-h-is-f} below which shows the conjecture is true for the independence
complexes of bipartite graphs.

Conjecture~\ref{con:h-is-f} was stated independently in~\cite[Conjecture~1.4]{CV} and further expanded in~\cite[Conjecture~1.5]{CV}.
A strengthening of Conjecture~\ref{con:h-is-f} in the case of a Gorenstein flag complex has been proposed in~\cite[Conjecture~1.4]{NP};
instances of this conjecture have been established in \cite{NPT}.

\section{Bipartite graphs} \label{sec:bipartite}

We now restrict ourselves to bipartite graphs and explore both the Buchsbaum and Cohen-Macaulay properties for the associated
independence complexes.

Let $G$ be a (non-empty) graph with vertex set $V$.  We call $G$ {\em bipartite} if $V$ can be partitioned into disjoint
sets $V_1$ and $V_2$, such that each is an independent set in $G$.  If $G$ is a bipartite graph with $\#V_1 = m$ and
$\#V_2 = n$, such that every vertex in $V_1$ is adjacent to every vertex in $V_2$, then $G$ is the {\em complete
bipartite graph} $\mathcal{K}_{m,n}$.

In the case of bipartite graphs, there are known results characterising when the independence complex of the graph is
pure or Cohen-Macaulay.
\begin{theorem}{\cite[Theorem~1.1]{Vi-2}} \label{thm:bi-pure}
    Let $G$ be a bipartite graph without isolated vertices.  Then $\Ind{G}$ is pure if and only if there is a partition
    $V_1 = \{x_1, \ldots, x_n\}$ and $V_2 = \{y_1, \ldots, y_n\}$ of the vertices of $G$ such that:
    \begin{enumerate}
        \item $x_iy_i$ is an edge of $G$, for all $1 \leq i \leq n$, and
        \item if $x_iy_j$ and $x_jy_k$ are edges in $G$, for $i, j$, and $k$ distinct, then $x_iy_k$ is an edge in $G$.
    \end{enumerate}
\end{theorem}
In this case, we call such a partition and ordering of the vertices a {\em pure order} of $G$.  Further we
will say that a pure order {\em has a cross} if, for some $i \neq j$, $x_iy_j$ and $x_jy_i$ are edges of $G$, otherwise
we say the order is {\em cross-free}.

\begin{theorem}{\cite[Theorem~3.4]{HH}} \label{thm:bi-cm}
    Let $G$ be a bipartite graph on a vertex set $V$ without isolated vertices.  Then $\Ind{G}$ is Cohen-Macaulay if and only
    if there is a pure ordering $V_1 = \{x_1, \ldots, x_n\}$ and $V_2 = \{y_1, \ldots, y_n\}$ of $G$, such that
    $x_iy_j$ being an edge in $G$ implies $i \leq j$.
\end{theorem}

Notice that if $G$ has isolated vertices $Z = \{z_1, \ldots, z_m\}$, then $\Ind{G}$ is pure (resp. Cohen-Macaulay) if and only if
$\Ind{(G \setminus Z)}$ is pure (resp. Cohen-Macaulay).

A rather direct consequence of Theorem~\ref{thm:bi-cm} is that the Cohen-Macaulayness of the independence complex of a bipartite graph
implies vertex-decomposability.
\begin{corollary}{\cite[Theorem~2.10]{VT}} \label{cor:bi-cm-vd}
    Let $G$ be a bipartite graph.  Then its independence complex $\Ind{G}$ is Cohen-Macaulay if and only if $\Ind{G}$ is vertex-decomposable.
\end{corollary}

Thus, the Cohen-Macaulayness of the independence complex of a bipartite graph also implies being squarefree glicci.  This
provides a nice class of examples of Cohen-Macaulay simplicial complexes that are squarefree glicci.
\begin{corollary} \label{cor:bi-glicci}
    Let $G$ be a bipartite graph.  If $\Ind{G}$ is Cohen-Macaulay, then $\Ind{G}$ is squarefree glicci.
\end{corollary}
\begin{proof}
    Since $\Ind{G}$ is vertex-decomposable by Corollary~\ref{cor:bi-cm-vd}, then $\Ind{G}$ is squarefree glicci~\cite[Theorem~3.3]{NR}.
\end{proof}

\subsection{Buchsbaum bipartite graphs} \label{sub:Buchsbaum}

In order to classify which bipartite graphs have Buchsbaum independence complexes, we need to find a new classification of 
bipartite graphs with Cohen-Macaulay independence complexes.  First, we see that for a bipartite graph with pure independence
complex, we only need to look at one pure order to determine if the graph is cross-free.  Hence, cross-free is a property of
the graph itself, rather than a property of a particular pure order.
\begin{lemma} \label{lem:bi-pure-xfree}
    Let $G$ be a bipartite graph with pure independence complex.  Then every pure order of $G$ has a cross if and only if some
    pure order of $G$ has a cross.
\end{lemma}
\begin{proof}
    Since we may look at each component individually, it is sufficient to consider connected graphs.  Let $G$ be a connected
    bipartite graph with a pure independence complex and let $\{x_1, \ldots, x_n\}$ and $\{y_1, \ldots, y_n\}$
    be a pure order of $G$ which has a cross, say $x_c y_d$ and $x_c y_d$. 

    As $G$ is connected, then the bi-partitioning of the vertex set is unique and so every pure order of $G$ is of the form 
    $\{x_{\al (1)}, \ldots, x_{\al (n)}\}$ and $\{y_{\be (1)}, \ldots, y_{\be (n)}\}$, for some permutations $\al$ and $\be$
    in $\mathfrak{S}_n$, the symmetric group on $n$ elements.  Notice then $x_i y_j$ is in $G$ if and only if 
    $x_{\al (i)}y_{\be (j)}$ is in $G$.

    If $\be^{-1}(\al(c))$ is $c$ or $d$, then clearly $x_d y_{\be^{-1}(\al(c))}$ is in $G$.  On the other hand, suppose $\be^{-1}(\al(c))$ is neither
    $c$ or $d$.  As the order is pure, $x_{\al(c)} y_{\al(c)}$ is in $G$ and so $x_c y_{\be^{-1}(\al(c))}$ is in $G$.  Further, since $x_d y_c$ is 
    in $G$ and the order is pure, then $x_d y_{\be^{-1}(\al(c))}$ is in $G$.  Thus, regardless of $\al$ and $\be$, we have that 
    $x_d y_{\be^{-1}(\al(c))}$ is in $G$ and so $x_{\al(d)} y_{\al(c)}$ is in $G$.  

    We can similarly show that $x_{\al(c)} y_{\al(d)}$ is in $G$.  Thus, the cross given by $x_{\al(c)} y_{\al(d)}$ and
    $x_{\al(d)}y_{\al(c)}$ is in $G$.
\end{proof}

\begin{corollary} \label{cor:bi-hasx-not-cm}
    Any bipartite graph with a pure order that has a cross has a non-Cohen-Macaulay independence complex.
\end{corollary}
\begin{proof}
    Let $G$ be a bipartite graph with a pure order that has a cross.  Then every pure order of $G$ has a cross by
    Lemma~\ref{lem:bi-pure-xfree}, hence has an edge $x_iy_j$ such that $j < i$.  Thus, by Theorem~\ref{thm:bi-cm},
    $\Ind{G}$ is not Cohen-Macaulay.
\end{proof}

Next, we see that a cross-free bipartite graph has two vertices of degree one.
\begin{lemma} \label{lem:bi-xfree-degree}
    Any cross-free bipartite graph with at least two vertices has, for any pure order, two vertices both of degree
    one that are in separate components of the bi-partition.
\end{lemma}
\begin{proof}
    It is sufficient to consider connected graphs; otherwise, we may look at each component individually.  Let $G$ be a cross-free
    connected bipartite graph and let $\{x_1, \ldots, x_n\}$ and $\{y_1, \ldots, y_n\}$ be a pure order of $G$.  If $n = 1$, then
    clearly $\deg_G{x_1} = \deg_G{y_1} = 1$.

    Suppose $n \geq 2$.  Let $H = G \setminus \{x_1, y_1\}$, then $H$ is also cross-free, and by induction has two vertices,
    say $x_i$ and $y_j$ of degree one.

    Assume $G$ has no vertices of degree one.  As $G$ has no vertices of degree one, if $i = j$, then $x_1y_i$ and $x_iy_1$
    are edges in $G$; this contradicts $G$ being cross-free.  Suppose then $i \neq j$.  Then $x_iy_1$ and $x_1y_j$ are edges
    in $G$.  As the order is pure, then $x_iy_j$ is in $G$ hence in $H$, but this contradicts $\deg_H{x_i} = 1$.

    Assume $G$ has vertices $\{x_{i_1}, \ldots, x_{i_m}\}$ of degree one, $m \geq 1$, but all vertices $\{y_1, \ldots, y_n\}$
    have degree at least two.  If $m = n$, then the $y_i$ must be connected, contradicting $G$ bipartite.  Let
    $J = G \setminus \{x_{i_1}, \ldots, x_{i_m}, y_{i_1}, \ldots, y_{i_m}\}$, then $J$ is also cross-free, and by induction has
    two vertices, say $x_i$ and $y_j$ of degree one.  But then $y_j$ must be connected to one of the $x_{i_t}$ which contradicts
    their having degree one.
\end{proof}

We summarise the above results to get a characterisation of bipartite graphs with Cohen-Macaulay independence complexes.
\begin{proposition} \label{pro:bi-xfree-is-cm}
    Let $G$ be a bipartite graph.  Then $G$ has a cross-free pure order if and only if its independence complex $\Ind{G}$ is Cohen-Macaulay.
\end{proposition}
\begin{proof}
    It is sufficient to consider connected graphs; otherwise, we may look at each component individually.
    Let $G$ be a cross-free connected bipartite graph and let $\{x_1, \ldots, x_n\}$ and $\{y_1, \ldots, y_n\}$ be
    a pure order of $G$.  If $n = 1$, then clearly $\Ind{G}$ is Cohen-Macaulay.

    Assume $n \geq 2$.  Then by Lemma~\ref{lem:bi-xfree-degree}, we may assume $y_1$ has degree one.  As
    $H = G \setminus \{x_1, y_1\}$ has a Cohen-Macaulay independence complex by induction, and all edges in $G$ not in
    $H$ are of the form $x_1y_i$, for some $1 \leq i \leq n$, then $\Ind{G}$ is also Cohen-Macaulay.

    If $G$ is has a cross, then by Corollary~\ref{cor:bi-hasx-not-cm}, $\Ind{G}$ is not Cohen-Macaulay.
\end{proof}

We will use the following classification of Buchsbaum complexes.  Note that the following theorem,
\cite[Theorem~3.2]{Sc}, has been rewritten using Reisner's Theorem.
\begin{theorem} \label{thm:complex-Buchsbaum}
    A simplicial complex is Buchsbaum if and only if it is pure and the link of each vertex is Cohen-Macaulay.
\end{theorem}

Finally, we can classify all bipartite graphs with Buchsbaum independence complexes, and surprisingly, find yet another
classification of bipartite graphs with Cohen-Macaulay independence complexes.  Also note, that this theorem was proven
independently, and in a different manner, in~\cite[Theorem~1.3]{HYZ}.
\begin{theorem} \label{thm:bi-Buchsbaum}
    Let $G$ be a bipartite graph.  Then its independence complex $\Ind{G}$ is Buchsbaum if and only if $G$ is a complete bipartite graph
    $\mathcal{K}_{n,n}$, for some $n$, or $\Ind{G}$ is Cohen-Macaulay.
\end{theorem}
\begin{proof}
    If $\Ind{G}$ is Cohen-Macaulay, then $\Ind{G}$ is Buchsbaum (this holds in general).  If $G = \mathcal{K}_{n,n}$, then for all
    vertices $v$ of $G$, $\link_{\Ind{G}}{v} = \Ind(G \setminus (v \cup N_G(v)))$ is a simplex on $n-1$ vertices, which is
    Cohen-Macaulay.  Thus, by Theorem~\ref{thm:complex-Buchsbaum}, $\Ind{G}$ is Buchsbaum.

    Suppose $G$ is not $K_{n,n}$ and $\Ind{G}$ is not Cohen-Macaulay.  Then by Proposition~\ref{pro:bi-xfree-is-cm}, every pure
    order of $G$ has a cross, say $x_1y_2, x_2y_1$.  Notice then $N_G(x_1) = N_G(x_2)$ and $N_G(y_1) = N_G(y_2)$, as $\Ind{G}$
    is pure.  Furthermore, if $N_G(y_1) = \{x_1, \ldots, x_n\}$ and $N_G(x_1) = \{y_1, \ldots, y_n\}$, then $G$
    is $\mathcal{K}_{n,n}$.  Hence we may assume there exists an $x \notin N_G(y_1)$.  Then
    $G \setminus (x \cup N_G(x))$ still contains $x_1y_2, x_2y_1$, hence does not have a Cohen-Macaulay independence
    complex by Proposition~\ref{pro:bi-xfree-is-cm}.  Therefore, $\Ind{G}$ is not Buchsbaum, by Theorem~\ref{thm:complex-Buchsbaum}.
\end{proof}

We note that the independence complex of ${\mathcal{K}_{n,n}}$ is Cohen-Macaulay if and only if $n = 1$.

\subsection{Compress and extrude} \label{sub:compress-extrude}
We now return to the question of whether the $h$-vector of a Cohen-Macaulay flag complex is a face vector
of a flag complex (see Conjecture~\ref{con:h-is-f}).

Let $G$ be a bipartite graph with a Cohen-Macaulay independence complex.  Then there exists some tri-partitioning 
$X = \{x_1, \ldots, x_n\}$, $Y = \{y_1, \ldots, y_n\}$, and $Z = \{z_1, \ldots, z_m\}$ such that:
\begin{enumerate}
    \item the vertices of $Z$ are exactly the isolated vertices of $G$,
    \item $X$ and $Y$ are a pure order of $G \setminus Z$, and
    \item $x_iy_j$ in $G$ implies $i \leq j$.
\end{enumerate}

Define the {\em compression} of $G$, denoted by $\check{G}$, by ``compressing'' all the edges $x_iy_i$ to the vertex $x_i$ and
removing the vertices of $Z$ altogether.  That is,
\[
    \check{G} := (X, \{x_ix_j \st i < j \mbox{ and } x_iy_j \in G \}).
\]
It turns out that compressing is the right action to find the desired simplicial complex.

\begin{proposition} \label{pro:bi-h-is-f}
    Let $G$ be a bipartite graph.  If $\Ind{G}$ is Cohen-Macaulay, then $h(\Ind{G}) = f(\Ind{\check{G}})$.
\end{proposition}
\begin{proof}
    Let $X, Y, Z$ be as above.  Notice that $\dim{\Ind{G}} = n+m-1$, as, e.g., $X \cup Z$ is a maximal independent
    set in $G$.

    If $n = 0$, then $G$ is a graph of $m$ disjoint vertices so $\check{G}$ is the empty graph and  $\Ind{\check{G}}$ 
    is the empty complex, hence $h(\Ind{G}) = (1) = f(\Ind{\check{G}})$.

    Suppose $n \geq 1$ and let $H = G \setminus (x_n \cup N_G(x_n))$; that is, $\Ind{H} = \link_{\Ind{G}}x_n$. 
    Then $\Ind{H}$ is Cohen-Macaulay and, further, $\dim{\Ind{H}} = n + m - 2$, as $(X \setminus x_n) \cup Z$ is
    a maximal independent set in $H$.  Thus, by induction, $h(\Ind{H}) = f(\Ind{\check{H}})$.  Using Equation~\eqref{equ:h-to-f}, we obtain
    \[
        f_{j-1}(\Ind{H}) = \sum_{i=0}^{j}\binom{n+m-1-i}{j-i}h_{i}(\Ind{H}) = \sum_{i=0}^{j}\binom{n+m-1-i}{j-i}f_{i-1}(\Ind{\check{H}}).
    \]
    Further, we have
    \[
        f_{j-1}(\Ind{H}) + f_{j-2}(\Ind{H}) = \sum_{i=0}^{j}\binom{n+m-i}{j-i}f_{i-1}(\Ind{\check{H}}),
    \]
    where $f_{j-1}(\Ind{H})$ counts the number of independent sets of size $j$ in $G$ without $x_n$ and $y_n$ and
    $f_{j-2}(\Ind{H})$ counts the number of independent sets of size $j$ in $G$ with $x_n$.

    Let $J = G \setminus (y_n \cup N_G(y_n))$; that is, $\Ind{J} = \link_{\Ind{G}}y_n$.  Then $\Ind{J}$ is Cohen-Macaulay,
    and, further, $\dim{\Ind{J}} = n + m - 2$, as $(Y \setminus y_n) \cup Z$ is a maximal independent set in $J$.  Thus, by
    induction, $h(\Ind{J}) = f(\Ind{\check{J}})$ and so
    \[
        f_{j-2}(\Ind{J}) = \sum_{i=0}^{j-1}\binom{n+m-1-i}{j-1-i}h_{i}(\Ind{J})= \sum_{i=0}^{j-1}\binom{n+m-1-i}{j-1-i}f_{i-1}(\Ind{\check{J}}),
    \]
    which counts the number of independent sets of size $j$ in $G$ with $y_n$, as the vertices of $J$ are
    exactly those independent of $y_n$.

    Now, $f_{i-1}(\Ind{\check{H}}) + f_{i-2}(\Ind{\check{J}}) = f_{i-1}(\Ind{\check{G}})$ as the first counts
    the independent sets of size $i$ in $\check{G}$ without $y_n$ and the second counts the independent sets of
    size $i$ in $\check{G}$ with $y_n$.  Hence, for $0 \leq j \leq n+m$,
    \begin{align*}
        f_{j-1}(\Ind{G}) & = f_{j-1}(\Ind{H}) + f_{j-2}(\Ind{H}) + f_{j-2}(\Ind{J}) \\
                         & = \sum_{i=0}^{j}\binom{n+m-i}{j-i}f_{i-1}(\Ind{\check{H}}) + \sum_{i=0}^{j-1}\binom{n+m-1-i}{j-1-i}f_{i-1}(\Ind{\check{J}}) \\
                         & = \sum_{i=0}^{j}\binom{n+m-i}{j-i}\left( f_{i-1}(\Ind{\check{H}}) + f_{i-2}(\Ind{\check{J}}) \right) \\
                         & = \sum_{i=0}^{j}\binom{n+m-i}{j-i}f_{i-1}(\Ind{\check{G}}).
    \end{align*}
    Solving this system for $f_{j-1}(\Ind{\check{G}})$ yields
    \[
        f_{j-1}(\Ind{\check{G}}) = \sum_{i=0}^{j}{(-1)^{j-i}\binom{n+m-i}{j-i}f_{i-1}(\Ind{G})} = h_j(\Ind{G}),
    \]
    for $0 \leq j \leq n+m$.
\end{proof}

The following consequence was independently, and in a different manor, shown in~\cite[Corollary~5.4]{CV}.
\begin{theorem} \label{thm:hbiCM-is-fFlag}
    The $h$-vector of an independence complex of a bipartite graph that is Cohen-Macaulay is a face vector of a flag complex.
\end{theorem}

A natural question is, which graphs are compressions of bipartite graphs with Cohen-Macaulay independence complexes?
To study this, we define an {\em extrusion} of any graph $G$ on vertices $\{v_1, \ldots, v_n\}$ to be a new bipartite graph
$\hat{G}$ on vertices $\{x_1, \ldots, x_n, y_1, \ldots, y_n\}$ with edge set given by the edges $x_iy_i$, for $1 \leq i \leq n$,
and either $x_iy_j$ or $x_jy_i$, for each edge $v_iv_j$ in $G$.  Note extrusion is {\em not} unique and if $\Ind{\hat{G}}$ is pure,
then $\hat{G}$ is cross-free by construction and thus $\Ind(\hat{G})$ is Cohen-Macaulay by Proposition~\ref{pro:bi-xfree-is-cm}.

We call a graph {\em Cohen-Macaulay extrudable} if there is some extrusion of the graph with a Cohen-Macaulay independence complex.
\begin{proposition} \label{pro:bi-cm-extrudable}
    Bipartite graphs are Cohen-Macaulay extrudable.
\end{proposition}
\begin{proof}
    Let $G$ be a bipartite graph with bi-partition $U = \{u_1, \ldots, u_m\}, V = \{v_{m+1}, \ldots, v_n\}$.  Then $G$ can
    be extruded to the graph $\hat{G}$ with vertex set $\{x_1, \ldots, x_n, y_1, \ldots, y_n\}$ and edge set given by
    $x_iy_i$, for $1 \leq i \leq n$, and $x_iy_j$, for all edges $u_iv_j$ in $G$.

    Hence, as $y_i$, for $1 \leq i \leq m$, and $x_i$, for $m+1 \leq i \leq n$, are all degree one, we have that 
    $\Ind{\hat{G}}$ is pure and thus Cohen-Macaulay.
\end{proof}

The extrusion described in the proof of Proposition~\ref{pro:bi-cm-extrudable} is just whiskering the graph at every vertex.
\begin{example} \label{exa:cm-extrusion}
    Figure~\ref{fig:extrude} demonstrates the extrusion described in the proof of Proposition~\ref{pro:bi-cm-extrudable}.  It illustrates
    that the chosen extrusion $\hat{G}$ is simply a whiskering of the graph $G$ at every vertex.
    \begin{figure}[!ht]
        \includegraphics{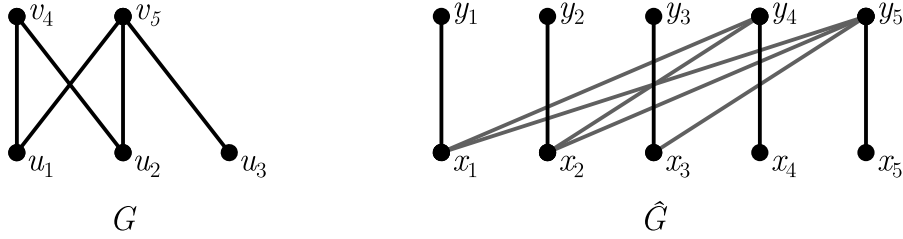}
        \caption{The extrusion $\hat{G}$ of the bipartite graph $G$ as described in the proof of Proposition~\ref{pro:bi-cm-extrudable}; the dark grey
            lines in the graph $\hat{G}$ correspond to the lines originating from the graph $G$.}
        \label{fig:extrude}
    \end{figure}
\end{example}

Hence by Proposition~\ref{pro:cw-h-is-f}, we have the following consequence.
\begin{corollary} \label{cor:bi-bi}
    The face vector of the independence complex of every bipartite graph is the $h$-vector of the independence of some
    vertex-decomposable bipartite graph.
\end{corollary}

Unfortunately, the converse is not true (see Example~\ref{exa:bi-bi-converse-fail} below).  However, by Proposition~\ref{pro:bi-h-is-f}
the $h$-vector of the independence complex of a bipartite graph that is vertex-decomposable is the face vector of the independence
complex of some (not necessarily bipartite) graph.
\begin{example} \label{exa:bi-bi-converse-fail}
    Let $G$ be the Ferrers graph given by the edges $x_1y_1$, $x_1y_2$, $x_1y_3$, $x_2y_2$, $x_2y_3$, $x_3y_3$; this graph
    is shown in figure~\ref{fig:bipartite}.
    Then $\Ind{G}$ is Cohen-Macaulay and hence vertex-decomposable, moreover, the $h$-vector of $\Ind{G}$ is $(1,3)$.  However,
    the only graph with independence complex having face vector $(1,3)$ is the three-cycle, which is not bipartite.
    \begin{figure}[!ht]
        \includegraphics{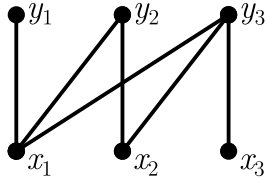}
        \caption{The Ferrers graph $G$}
        \label{fig:bipartite}
    \end{figure}
\end{example}

A graph is said to have {\em odd-holes} if it has an induced subgraph that is an odd-cycle of length at least five.
\begin{proposition} \label{pro:odd-holes-not-cm-extrudable}
    A graph with odd-holes is not Cohen-Macaulay extrudable.
\end{proposition}
\begin{proof}
    Let $G$ be a graph with odd-holes.  In particular, let $v_1, \ldots, v_{2m+1}$ be an odd-hole of $G$ and suppose
    that $v_i$ extrudes to the pair $x_i, y_i$.  Without loss of generality, assume the edge $v_1v_2$ is extruded to $x_1y_2$. 
    Then we must extrude the edge $v_2v_3$ to $x_3y_2$, otherwise we introduce an obstruction to purity
    as $v_1v_3$ is not an edge in $G$.  Continuing in this manner, edges $v_{2k-1}v_{2k}$ extrude to $x_{2k-1}y_{2k}$ 
    and edges $v_{2k}v_{2k+1}$ extrude to $x_{2k+1}y_{2k}$.

    The remaining edge $v_1v_{2m+1}$ can either be extruded as $x_1y_{2m+1}$ or $x_{2m+1}y_1$.  In the first case,
    then $x_1y_{2m+1}$ and $x_{2m+1}y_{2m}$ would require $x_1y_{2m}$ for purity, but this contradicts our
    choice of an odd-hole, as this would require $v_1v_{2m}$ to be in $G$.  Similarly for the second case.
    Thus, odd-holes obstruct pure extrusions.
\end{proof}

Notice that odd-hole-free graphs without triangles are bipartite, hence Cohen-Macaulay extrudable.  However,
not all odd-hole-free graphs with triangles are Cohen-Macaulay extrudable.
\begin{example} \label{exa:cm-extrudable}
    Let $G$ be the three-cycle on vertices $\{u,v,w\}$.  Then the extrusion
    $\hat{G} = (\{u,v,w,x,y,z\}, \{ux, vy, wz, uy, vz, uz\})$ has a Cohen-Macaulay independence complex.

    Let $H = (\{u,v,w,x,y,z\}, \{uv, uw, vw, xy, xz, yz, ux, uy, vz\})$, then one can use a program such as Macaulay2~\cite{M2}
    to check the $2^9 = 512$ possible extrusions of $H$ to see that it is not Cohen-Macaulay extrudable.
\end{example}

Hence, we close with a question: which odd-hole-free graphs with triangles are Cohen-Macaulay extrudable?

\begin{acknowledgement}
    The authors would like to acknowledge the extensive use of Macaulay2 \cite{M2} in testing conjectures and generating
    examples.  The authors would like to thank Ben Braun for helpful discussions and comments and the anonymous referees 
    for many insightful comments.
\end{acknowledgement}



\begin{thebibliography}{9}

\bibitem{BN}
    E.\ Babson, I.\ Novik,
    {\em Face numbers and nongeneric initial ideals},
    Electron.\ J.\ Combin.\ {\bf 11} (2006), R\#25, 23 pp.

\bibitem{BFS}
    A.\ Bj\"orner, P.\ Frankl, R.\ Stanley,
    {\em The number of faces of balanced Cohen-Macaulay complexes and a generalized Macaulay theorem},
    Combinatorica {\bf 7} (1987), 23--34.

\bibitem{CV}
    A.\ Constaninescu, M.\ Varbaro,
    {\em On the $h$-vectors of Cohen-Macaulay Flag Complexes},
    Preprint (2010); also available at arXiv:1004.0170.

\bibitem{Co}
    D.\ Cook~II,
    {\em Simplicial Decomposability},
    The Journal of Software for Algebra and Geometry {\b 2} (2010), 20--23.

\bibitem{DE}
    A.\ Dochtermann, A.\ Engstr\"om,
    {\em Algebraic properties of edge ideals via combinatorial topology},
    Electron.\ J.\  Combin.\ {\bf 16} (2009), \#R2, 24 pp.

\bibitem{E}
    J.\ Eckhoff,
    {\em Intersection properties of boxes. I. An upper-bound theorem},
    Israel J.\ Math.\ {\bf 62} (1988), 283--301.

\bibitem{En}
    A.\ Engstr\"om,
    {\em Independence complexes of claw-free graphs},
     European J.\ Combin.\ {\bf 29} (2008), 234--241.

\bibitem{FH}
    C.\ A.\ Francisco, H.\ T.\ H\`a,
    {\em Whiskers and sequentially Cohen-Macaulay graphs},
    J.\ Comb.\ Theory Ser.\ A {\bf 115} (2008), 304--316.

\bibitem{Fr}
    A.\ Frohmader,
    {\em Face vectors of flag complexes},
    Israel J.\  Math. {\bf 164} (2008), 153--164.

\bibitem{HYZ}
    H.\ Haghighi, S.\ Yassemi, R.\ Zaare-Nahandi,
    {\em Bipartite $S_2$ graphs are Cohen-Macaulay},
    Bull.\ Math.\ Soc.\ Sci.\ Math.\ Roumanie {\bf 53} (2010), 125--132.

\bibitem{HH}
    J.\ Herzog, T.\ Hibi,
    {\em Distributive lattices, bipartite, graphs, and Alexander duality},
    J.\ Algebraic Comb. {\bf 22} (2005), 289--302.

\bibitem{HU}
    C.\ Huneke, B.\ Ulrich,
    {\em The structure of linkage},
    Ann.\ of Math. (2) {\bf 126} (1987), 277--334.

\bibitem{LM}
    V.\ Levit, E.\ Mandrescu,
    {\em On the roots of independence polynomials of almost all very well-covered graphs},
    Discrete Appl.\ Math.\ {\bf 156} (2008), 478--491.

\bibitem{M2}
    D.\ Grayson, M.\ Stillman,
    {\em Macaulay2, a software system for research in algebraic geometry},
    available at {\tt http://www.math.uiuc.edu/Macaulay2/}.

\bibitem{NR}
    U.\ Nagel, T.\ R\"omer,
    {\em Glicci simplicial complexes},
    J.\ Pure Appl.\ Algebra {\bf 212} (2008), 2250--2258.

\bibitem{NP}
    E.\ Nevo, T.\ K.\ Petersen,
    {\em On $\gamma$-vectors satisfying the Kruskal-Katona inequalities},
    Discrete Comput.\ Geom.\ {\bf 45} (2010), 503--521.

\bibitem{NPT}
    E.\ Nevo, T.\ K.\ Petersen, B.\ E.\ Tenner,
    {\em The $\gamma$-vector of a barycentric subdivision},
    J.\ Combin.\ Theory Ser.\ A {\bf 118} (2011), 1364--1380.

\bibitem{PB}
    J.\ S.\ Provan, L.\ J.\ Billera,
    {\em Decompositions of Simplicial Complexes Related to Diameters of Convex Polyhedra},
    Math. Oper. Res. {\bf 5} (1980), 576--594.

\bibitem{Sc}
    P.\ Schenzel,
    {\em On the number of faces of simplicial complexes and the purity of Frobenius},
    Math.\ Zeit. {\bf 178} (1981), 125--142.

\bibitem{SVV}
    A.\ Simis, W.\ Vasconcelos, R.\ Villarreal,
    {\em On the Ideal Theory of Graphs},
    J.\ Algebra {\bf 167} (1994), 389--416.

\bibitem{St}
    R.\ Stanley,
    {\em Combinatorics and Commutative Algebra},
    2nd edition. Progress in Mathematics {\bf 41} Birkh\"auser Boston, Inc. Boston, MA, 1996.

\bibitem{VT}
    A.\ Van Tuyl,
    {\em Sequentially Cohen-Macaulay bipartite graphs: vertex decomposability and regularity},
    Arch.\ Math.\ (Basel) {\bf 93} (2009), 451--459.

\bibitem{Vi}
    R.\ H.\ Villarreal,
    {\em Cohen-Macaulay graphs},
    Manuscripta Math. {\bf 66} (1990), 277--293.

\bibitem{Vi-2}
    R.\ H.\ Villarreal,
    {\em Unmixed bipartite graphs},
    Rev.\ Colombiana Mat.\ {\bf 41} (2007), 393--395.

\bibitem{Wo}
    R.\ Woodroofe,
    {\em Vertex decomposable graphs and obstructions to shellability},
    Proc.\ Amer.\ Math.\ Soc. {\bf 137} (2009), 3235--3246.

\end{thebibliography}
\end{document}